\documentclass[journal]{IEEEtran}

\usepackage{amsmath, amssymb, amsthm}
\usepackage{mathtools}
\usepackage{mathrsfs}
\usepackage{bm}
\usepackage{cite}
\usepackage{balance}
\usepackage{graphicx}
\usepackage{xcolor}
\usepackage[hidelinks]{hyperref}
\newtheorem{theorem}{Theorem}

\newtheorem{assumption}[theorem]{Assumption}

\theoremstyle{definition}
\newtheorem{definition}[theorem]{Definition}

\theoremstyle{problem}
\newtheorem{problem}[theorem]{Problem}

\theoremstyle{remark}

\title{Noise-Induced Navigation in Non-convex Domains \\
	and Compact Manifolds}

\author{Karthik Elamvazhuthi
	\thanks{Karthik Elamvazhuthi is an independent researcher (e-mail: karthikevaz@gmail.com)}}
\begin{document}
	\maketitle
	
	\begin{abstract}
		In this note, we study the problem of designing a feedback law that globally steers a system to a prescribed target configuration. Even if the system is fully actuated, topological obstructions generally prevent the existence of globally asymptotically stabilizing continuous feedback laws. We revisit this problem in a stochastic setting by allowing noise to enter through the control channels. Using a criterion for asymptotic stability in the large that combines local Lyapunov stability with positive recurrence, we constructively show that one can construct elementary feedback laws that achieve global asymptotic stability in the large, of the target equilibrium point in connected Euclidean domains with obstacles and manifolds without boundary. For Euclidean domains with obstacles, we also show that the method extends naturally
		to the problem of finding the minimizer of a strongly convex function with non-convex constraints. Numerical
		experiments illustrate the effectiveness of the approach for Euclidean domains with circular obstacles and the two dimensional sphere. Additionally, we study the role noise strength when there is a non-convex obstacle, in which case the system might show metastability.
	\end{abstract}
	
	\section{Introduction}

	A central objective in motion planning is the design of a feedback law that drives a system from any initial configuration to a desired goal $x_{eq}$, with global guarantees, while remaining in a prescribed compact domain. A number of algorithms have been developed to address this problem in robotics. However, even if the system is fully actuated, it is well known that the topology of the domain presents
	an obstruction to global asymptotic stability. On a compact domain with obstacles, no smooth
	feedback law can globally stabilize a target equilibrium~\cite{koditschek1990robot},
	and an analogous obstruction holds on compact manifolds ~\cite{bhat2000topological}. A number of approaches have been considered in literature to avoid such topological obstructions. For example, using discontinuous feedback laws \cite{malisoff2006global} combined with notion of sample and hold solutions, or using hybrid dynamics \cite{ochoa2025robust}.

	In this paper, we consider an approach that provides an alternative perspective to avoid these geometric obstructions by introducing noise into the closed-loop dynamics. It formalizes the intuition that one can use noise to help the system escape undesirable critical points. Rather than seeking a deterministic vector field, we design a controlled diffusion whose sample paths converge to $x_{eq}$ in a probabilistic sense. We enforce invariance of the domain by formulating the dynamics as a reflected stochastic differential equation \cite{pilipenko2014introduction}. Within this framework, we show that for any connected open subset of Euclidean space with smooth boundary, a prescribed target equilibrium can be smoothly stabilized, provided one adopts the stochastic notion of {\it asymptotic stability in the large}. We further show that the same construction extends to smooth stabilization on a compact Riemannian manifold. The idea of using noise to circumvent topological obstructions has been considered in \cite{nishimura2016stabilization} under much more general context than we study, of stabilizing underactuated systems. However, the solvability of the stabilization problem is contingent on being able to find a Lyapunov function. In the setting closer to what we consider is \cite{paternain2019stochastic}, where discrete-time stochastic gradient descent is used to find the minimizer of a strongly convex function, navigating around convex obstacles. In contrast to our approach, this work requires construction of a navigation function. 
	
	The construction of our feedback law is based on a criterion for stability due to  \cite{khasminskii2011stochastic}, where the author establishes a sufficient condition for asymptotic stability in the large in the form of a local Lyapunov condition near the equilibrium, combined with positive recurrence of a neighborhood of the equilibrium. The choice of noise in our construction is done in the following manner: i) Near $x_{eq}$, the deterministic stabilizing dynamics dominates the behavior so that a quadratic Lyapunov function is locally non-increasing along the generator, yielding local stability in probability, ii)  Away from $x_{eq}$, the noise renders the generator uniformly elliptic, which in turn yields the recurrence needed to drive the process back into a neighborhood of the goal.

	\section{Euclidean Domains with Boundary}
	First, we consider the case when the domain $\Gamma$ is a bounded open subset of $\mathbb{R}^N$ with smooth 
	boundary. Consider the reflected diffusion process \cite{pilipenko2014introduction}
	
	Let $\Gamma \subset \mathbb{R}^d$ be a bounded connected open set with smooth
	boundary, and let $n(x)$ denote the inward unit normal vector at
	$x \in \partial \Gamma$. Let
	$(\Omega,\mathcal{F},\{\mathcal{F}_t\}_{t\ge 0},\mathbb{P})$
	be a filtered probability space satisfying the usual conditions, and let
	$ W(t) = \bigl(W^1(t),\ldots,W^d(t)\bigr) $
	be a standard $\mathbb{R}^d$-valued $\{\mathcal{F}_t\}$-Brownian motion.
	Given smooth coefficient functions
	$
	u_i,a_i \in C^\infty(\overline{\Gamma}), ~ i=1,\ldots,d,
	$
	we consider the reflected Stratonovich SDE \cite{lions1984stochastic, ikeda2014stochastic, pilipenko2014introduction}
	\begin{equation}\label{eq:sde_eucl}
		\begin{aligned}
			dX(t) &= \sum_{i=1}^d u_i(X(t))A_i\,dt+\sum_{i=1}^d a_i(X(t))A_i\circ dW^i(t)\\
			&\quad+n(X(t))\,d\varphi(t).
		\end{aligned}
	\end{equation}
	with initial condition $X(0)=X_0 \in \Gamma$. Here
	$ A_i=\partial_{x_i}, ~ i=1,\ldots,d,
	$
	are the standard coordinate vector fields on $\mathbb{R}^d$. The process
	$X(t)$ takes values in $\overline{\Gamma}$, and $\varphi(t)$ is a continuous,
	nondecreasing, adapted process with $\varphi(0)=0$ that increases only when
	$X(t)$ lies on the boundary, in the sense that $
	\int_0^t \mathbf{1}_{\Gamma}(X(s))\,d\varphi(s)=0,
	~ t\ge 0.$
	Thus the term $n(X(t))\,d\varphi(t)$ enforces reflection at $\partial\Gamma$
	and keeps the trajectory inside $\overline{\Gamma}$. Under the hypothesis existence and uniqueness of solutions is well known \cite{lions1984stochastic, ikeda2014stochastic, pilipenko2014introduction}. 
	
	Before we state our objective, we recall a relevant notion of stability
	from~\cite{khasminskii2011stochastic}.
	
	\begin{definition}[Stability in probability]
		A solution $X(t) \equiv 0$ of~\eqref{eq:sde_eucl} is said to be \emph{stable in probability} for $t \geq 0$ if for any $s \geq 0$ and $\varepsilon > 0$,
		\[
		\lim_{x \to 0} \mathbb{P}\left\{ \sup_{t > s} |X^{s,x}(t)| > \varepsilon \right\} = 0.
		\]
	\end{definition}
	
	\begin{definition}[Asymptotic stability in the large]
		The solution $X(t) \equiv 0$ of~\eqref{eq:sde_eucl} is said to be (asymptotically) stable in the large if it is stable in probability and, in addition, for all $s$ and $x$,
		\[
		\mathbb{P}\left\{ \lim_{t \to \infty} X^{s,x}(t) = 0 \right\} = 1.
		\]
	\end{definition}

	Given these definitions, our objective is the following.
	
	\medskip
	
	\begin{problem}Find $u_i(x)$ and $a_i(x)$ such that 
		a given point $x_{eq} \in \Gamma$ is asymptotically stable in 
		the large, for the process \eqref{eq:sde_eucl}.
	\end{problem}

	\section{Analysis}
	The following theorem addresses the problem defined in the previous section.
	
	\begin{theorem}
		There exist smooth functions $a_i, u_i \in C^\infty(\overline{\Gamma})$ 
		such that the equilibrium point $x_{eq}$ of the process 
		\eqref{eq:sde_eucl} is asymptotically stable in probability in 
		the large. Moreover, a choice of the controls is 
		\begin{align}
			&a_i(x) = |x-x_{\rm eq}|^2, \quad u_i(x) = -\alpha (x_i-x_{\rm eq,i}), 
			\qquad i = 1,\ldots,d.
		\end{align}
		for a suitable $\alpha>0$.
	\end{theorem}
	
	\begin{proof}
		Without loss of generality take $x_{eq} = 0$. Define $
		V(x) = \frac{1}{2}\sum_{i=1}^d x_i^2$.
		The parameter $\alpha > 1$ in the control laws is to be determined. The infinitesimal generator 
		of \eqref{eq:sde_eucl} is locally
		\begin{align*}
			LV(x)&=\sum_{i=1}^N u_i(x)\,\partial_{x_i}V
			+\frac12\sum_{i=1}^N a_i^2\,\partial_{x_i}^2V\\
			&\quad+\frac12\sum_{i=1}^N a_i(\partial_{x_i}a_i)(\partial_{x_i}V).
		\end{align*}
		Since $A_i = \partial_{x_i}$ and $a_i = |x|^2$ for $i = 1,\ldots,d$, 
		each term computes as:
		For $i=1,\ldots,d$,
		\[
		\partial_{x_i}V=x_i,\qquad
		\partial_{x_i}a_i=2x_i,\qquad
		\partial_{x_i}^2V=1.
		\]
		Substituting gives
		\begin{align*}
			LV
			&= \sum_{i=1}^d (-\alpha x_i)(x_i)
			+ \frac{1}{2}\sum_{i=1}^d 
			\left[|x|^4 \cdot 1 + |x|^2(2x_i)(x_i)\right] \\
			&= -\alpha\sum_{i=1}^d x_i^2 + \frac{d}{2}|x|^4 + |x|^2\sum_{i=1}^d x_i^2 \\
			&= -\alpha |x|^2 + \frac{d+2}{2}|x|^4 \\
			&= -2\alpha V + 2(d+2)V^2.
		\end{align*}
		Therefore, if $U$ is chosen sufficiently small so that
		\[
		V(x) \le \frac{\alpha}{2(d+2)}
		\qquad \text{for all } x \in U,
		\]
		then on $U\setminus\{x_{eq}\}$ we have
		\[
		LV \le -\alpha V < 0.
		\]
		By~\cite[Theorem~5.3]{khasminskii2011stochastic}, this implies that $x_{eq}$
		is stable in probability. Since $V$ has an infinitesimal upper limit,
		stability in probability is in fact uniform
		(\cite[Remark~5.3]{khasminskii2011stochastic}).
		
		Next, we prove recurrence. Let $B_r = B_r(x_{eq})$ be a closed ball with $\overline{B_r} \subset U$, 
		and set $D_r = \Gamma \setminus B_r$. The second order part of $L$ is 
		\[
		\frac{1}{2}\sum_{i=1}^d a_i^2(x)\,(\xi^i)^2
		= \frac{1}{2}|x|^4 \sum_{i=1}^d (\xi^i)^2.
		\]
		By construction, $a_i(x) = |x|^2 > 0$ on all of $D_r$ for $i = 1,\ldots,d$. 
		Since $\overline{D_r}$ is compact and each $a_i$ is continuous, $
		\inf_{x \in D_r} a_i(x) \geq \lambda > 0$
		for some constant $\lambda > 0$. Therefore,
		\[
		\frac{1}{2}\sum_{i=1}^d a_i^2(x)\,(\xi^i)^2
		\geq \frac{\lambda^2}{2}\sum_{i=1}^d (\xi^i)^2
		= \frac{\lambda^2}{2}\,|\xi|^2,
		\]
		so $L$ is uniformly elliptic on $D_r$. Consider the mixed Dirichlet--Neumann boundary value problem
		\[
		\begin{cases}
			Lu_r(x) = -1,              & x \in D_r, \\[0.4em]
			u_r(x) = 0,                & x \in \partial B_r, \\[0.4em]
			\dfrac{\partial u_r}{\partial n}(x) = 0, & x \in \partial\Gamma.
		\end{cases}
		\]
		
		Then from classical regularity theory  for mixed elliptic boundary value problems, there exists a unique weak solution \(u_r\) to this problem. See \cite[Section 9.6]{lions2012non}. Moreover, since the Dirichlet and Neumann parts of the boundary are separated by a non-zero distance, standard boundary regularity theory applies locally near each boundary point after flattening the boundary. In particular, combining the Euclidean half-space regularity results for the pure Dirichlet and pure Neumann problems with a  standard localization argument, one concludes that \(u_r\in C^\infty(\overline{D_r})\). See \cite{gilbarg1983trudinger,agranovich2015sobolev}.
		
		From \cite[Theorem 3.1.3]{pilipenko2014introduction}, $u_r(x) = \mathbb{E}_x[\tau_r]$ 
		is the expected first hitting time of $B_r$. Combining uniform stability in probability with the positive recurrence of
		$B_r$ established above, \cite[Theorem~5.7]{khasminskii2011stochastic} yields
		that $x_{eq}$ is asymptotically stable in probability in the large. We note that Khasminskii's results are stated for diffusions on $\mathbb{R}^\ell$
		without boundary, whereas our process is a reflected diffusion on $\Gamma$.
		However, the proof of \cite[Theorem~5.7]{khasminskii2011stochastic} relies only on the strong Markov property.
		For reflected diffusion~\eqref{eq:sde_eucl}, it is the unique solution of the
		submartingale problem, which yields the strong Markov property
		(see~\cite{pilipenko2014introduction}).
	\end{proof}

	\paragraph{Extension to general potential functions} The same construction as in the proof admits a generalization using an arbitrary
	smooth potential function. Let
	$V \in C^\infty(\bar{\Gamma})$ satisfy
	\[
	V(x_{eq})=0, \qquad V(x)>0
	\quad\text{for }x\in\overline{\Gamma}\setminus\{x_{eq}\},
	\]
	and assume that Hessian of $V$ at $x_{eq}$is positive definite. Then, after
	translating coordinates so that $x_{eq}=0$, one may choose
	\[
	a_i(x)=2V(x), \qquad
	u_i(x)=-\alpha\,\partial_{x_i}V(x),
	\qquad i=1,\dots,d,
	\]
	for all $x\in \Gamma$.

	For this choice,
	\[
	LV
	=
	\sum_{i=1}^d u_i\,\partial_{x_i}V
	+
	\frac12\sum_{i=1}^d
	\Bigl[
	a_i^2\,\partial_{x_i}^2V
	+
	a_i(\partial_{x_i}a_i)(\partial_{x_i}V)
	\Bigr].
	\]
	Since $a_i=2V$ and $\partial_{x_i}a_i = 2\,\partial_{x_i}V$, we obtain
	\begin{align*}
		LV&=-\alpha\sum_{i=1}^d(\partial_{x_i}V)^2
		+\frac12\sum_{i=1}^d(2V)^2\partial_{x_i}^2V\\
		&\quad+\frac12\sum_{i=1}^d(2V)(2\partial_{x_i}V)(\partial_{x_i}V)\\
		&=-\alpha|\nabla V|^2+2V|\nabla V|^2+2V^2\Delta V.
	\end{align*}
	Because at $x_{eq}$ the function $V$ has a positive definite Hessian, there exists a smaller
	neighborhood $U$ and constants $c,C>0$ such that
	\[
	c\,V(x) \le |\nabla V(x)|^2 \le C\,V(x),
	\qquad
	|\Delta V(x)| \le C,
	\qquad x \in U,
	\]
	where the lower bound is known as the Polyak-Lojasiewicz (PL) inequality. 
	Therefore,
	\[
	LV \le -\alpha c\,V + C_1 V^2
	\]
	for some constant $C_1>0$. After shrinking $U$ if necessary, it follows that
	\[
	LV \le -c_0 V < 0
	\qquad \text{on } U\setminus\{x_{eq}\},
	\]
	for some $c_0>0$, using the fact that $V$ is locally bounded. Thus the same local stability argument applies. The
	recurrence argument is unchanged, since $a_i=2V$ is strictly positive and
	therefore, by compactness, bounded away from zero on
	$\overline{\Gamma\setminus B_r(x_{eq})}$ for any
	closed ball $\overline{B_r(x_{eq})}\subset U$.
	
	This generalization will be the basis of our extension of the result to
	the case of compact manifolds.
	
	\paragraph{Finding minimizer of a strongly convex function.}
	Before we proceed to the next section, we consider the case where the objective is not merely to
	stabilize a prescribed point, but rather to find the minimizer of a given strongly convex function. This is the setting considered in \cite{paternain2019stochastic}. The discussion in the previous section is not directly applicable since it requires that the potential, and hence the noise, vanishes at the target equilibrium. If we don't know where the minimizer of a potential is, this is problematic. 
	
	To address this more general situation, instead of scaling the noise by the value of the Lyapunov function itself, we can scale it by
	$|\nabla V(x)|^2$. Since $V$ is strongly convex, its gradient vanishes only at the unique minimizer.
	Consequently, the diffusion degenerates at the minimizer, so that the minimizer remains an
	equilibrium, while away from the minimizer the diffusion is nondegenerate. This gives the same
	combination as before: local Lyapunov stability near the minimizer and recurrence from the rest
	of the domain.
	
	The idea is to choose the deterministic part of the dynamics as a gradient vector field,
	while choosing the noise amplitude to vanish at the minimizer. In particular, we take
	\[
	u_i(x)=-\alpha\partial_{x_i}V(x),\qquad
	a_i(x)=\sigma|\nabla V(x)|^2,
	\]
	for $i=1,\ldots,d$.
	Since \(x_\ast\) is an interior minimizer, \(\nabla V(x_\ast)=0\), and hence both the drift and
	diffusion vanish at \(x_\ast\). To check the Lyapunov estimate near \(x_\ast\), define 
	\[
	W(x):=V(x)-V(x_\ast).
	\]
	Then \(W(x_\ast)=0\), \(\nabla W=\nabla V\), and \(LW=LV\). The generator applied to \(W\) gives
	\begin{align*}
		LW&=-\alpha|\nabla V|^2+\frac{\sigma^2}{2}|\nabla V|^4\Delta V\\
		&\quad+\sigma^2|\nabla W|^2\nabla W^\top\nabla^2W\nabla W.
	\end{align*}
	The first term is the ``good term''. Since \(V\) is strongly convex near \(x_\ast\), there is a
	constant \(c>0\) such that
	\[
	|\nabla W(x)|^2 \geq cW(x)
	\]
	in a sufficiently small neighborhood of \(x_\ast\) with smooth boundary. Moreover, since \(W\) is smooth, \(\Delta W\)
	and \(\nabla^2W\) are locally bounded. Therefore, for \(x\) near \(x_\ast\),
	\[
	LW
	\leq
	-\alpha |\nabla V|^2
	+
	C\sigma^2 |\nabla V|^4 .
	\]
	Shrinking the neighborhood if necessary, the higher-order term is dominated by the negative
	gradient term, and we obtain
	\[
	LW
	\leq
	-\frac{\alpha}{2}|\nabla W|^2
	\leq
	-c_0 W
	\]
	for some \(c_0>0\). Thus \(W=V-V(x_\ast)\) satisfies the same local Lyapunov inequality near the
	minimizer as in the stabilization argument. Away from the minimizer, \(|\nabla V|\) is bounded
	below, so the diffusion becomes uniformly elliptic on compact subsets of
	\(\Gamma\setminus\{x_\ast\}\), giving the recurrence needed for asymptotic stability in the large.
	Thus, we can state the following result.

	\begin{theorem}
		Let $\Gamma \subset \mathbb{R}^d$ be a bounded connected open set with smooth boundary, and let $V \in C^\infty(\mathbb{R}^d)$ satisfy
		$ \nabla^2 V(x) \geq \mu I $
		for some $\mu>0$ and all $x \in \mathbb{R}^d$. Assume that \(V\) has a unique minimizer
		\(x_\ast \in \Gamma\). Then the minimizer $x_\ast$ is asymptotically stable
		in probability in the large for the reflected diffusion \eqref{eq:sde_eucl} with the choice of controls,
		\[
		a_i(x)=\sigma|\nabla V(x)|^2,\qquad
		u_i(x)=-\alpha\partial_{x_i}V(x),
		\]
		for $i=1,\ldots,d$, where $\alpha>0$ and $\sigma>0$ are constants.
	\end{theorem}

		

	
	\section{Manifold}
	In this section, we consider the generalization to manifolds.	Let $(M, g)$ be a connected compact smooth Riemannian manifold without boundary 
	of dimension $d$ with metric $g: TM \times TM \rightarrow \mathbb{R}$, and let $\mathfrak{X}(M)$ denote the space of smooth 
	vector fields on $M$.
	
	Once again, let $(\Omega, \mathcal{F}, \{\mathcal{F}_t\}_{t\geq 0}, \mathbb{P})$ be a 
	filtered probability space satisfying the usual conditions, and let 
	$W(t) = (W^1(t),\ldots,W^b(t))$ be a standard $\mathbb{R}^m$-valued 
	$\{\mathcal{F}_t\}$-Brownian motion. Given vector fields 
	$A_1,\ldots,A_m \in \mathfrak{X}(M)$ and smooth coefficient functions 
	$u_i, a_i \in C^\infty(M)$, we consider the following Stratonovich SDE 
	on $M$:
	\begin{equation}
		\label{eq:sdeman}
		\begin{aligned}
			dX(t)&=\sum_{i=1}^m u_i(X(t))A_i(X(t))\,dt\\
			&\quad+\sum_{i=1}^m a_i(X(t))A_i(X(t))\circ dW^i(t),
		\end{aligned}
	\end{equation}
	with $X(0)=X_0\in M$.
	Solutions of such SDEs are known to exist and unique under our smoothness assumptions \cite{ikeda2014stochastic}. 
	The generator of this process $L$ acts on smooth functions on $M$, and is given by:
	$$L f(x) = u_i A_i f + \frac{1}{2} (a_i A_i)^2 f  .$$
	We will need the manifold analogue of stochastic stability.
	\begin{definition}[Stability in probability]
		An equilibrium $x_{eq} \in M$ of~\eqref{eq:sdeman} is said to be \emph{stable in probability} for $t \geq 0$ if for any $s \geq 0$ and $\varepsilon > 0$,
		\[
		\lim_{x \to x_{eq}} \mathbb{P}\left\{ \sup_{t > s} d_g\bigl(X^{s,x}(t), x_{eq}\bigr) > \varepsilon \right\} = 0.
		\]
	\end{definition}
	
	\begin{definition}[Asymptotic stability in the large]
		An equilibrium $x_{eq} \in M$ of~\eqref{eq:sdeman} is said to be \emph{(asymptotically) stable in the large} if it is stable in probability and, in addition, for all $s \geq 0$ and $x \in M$,
		\[
		\mathbb{P}\left\{ \lim_{t \to \infty} d_g\bigl(X^{s,x}(t), x_{eq}\bigr) = 0 \right\} = 1.
		\]
	\end{definition}
	
	We will make the following assumptions on the process, which are a natural generalization of the setting considered in the last section.

	\begin{assumption}
		\label{asmp}
		The vector fields $\{A_1,\ldots,A_m\}$ satisfy
		$\mathrm{span}\{A_1(x),\ldots,A_m(x)\} = T_xM$
		for all $x \in M$, with possibly $m >d$.
	\end{assumption}
	
	The necessity of considering the possibility that $m > d$ is that, in general, a manifold might not be {\it globally parallizable}. For example, the two dimesnsional sphere $S^2$. On the other hand, if the manifold is a compact Lie group, then one can take $m = d$, and there exist a choice of orthonormal vector fields $\{A_1,\ldots,A_d\}$ that satisfy
	$\mathrm{span}\{A_1(x),\ldots,A_d(x)\} = T_xM$ for all $x \in M$.
	
	Given these definitions we have the following result.

	\begin{theorem}
		There exist smooth functions $a_i, u_i \in C^\infty(M)$ such that the
		equilibrium point $x_{eq}$ of the process \eqref{eq:sdeman} is
		asymptotically stable in probability in the large.
	\end{theorem}
	
	\begin{proof}
		There exists a smooth function $V \in C^\infty(M)$ such that
		\[
		V(x) \ge 0, \qquad V^{-1}(0)=\{x_{eq}\},
		\]
		and the Hessian of $V$ is positive definite at $x_{\rm eq}$.
		By the Morse lemma \cite[Lemma 2.2]{milnor1963morse}, there exist local coordinates
		$(x_1,\ldots,x_d)$ centered at $x_{eq}$, defined on a neighborhood $U$ of
		$x_{eq}$, such that $
		V(x)=\frac{1}{2}\sum_{j=1}^d x_j^2$
		on $U$. Write
		\[
		A_i=\sum_{j=1}^d c_i^j(x)\,\partial_{x_j},
		\qquad i=1,\ldots,m.
		\]
		Choose
		\[
		a_i(x)=2V(x), \qquad
		u_i(x)=-\alpha\,A_iV(x),
		\qquad i=1,\dots,m,
		\]
		for all $x\in M$.
		
		Here $\alpha>0$ is to be
		determined. Since $dV(x_{eq})=0$ and $V(x_{eq})=0$, we have
		\[
		u_i(x_{eq})=0, \qquad a_i(x_{eq})=0, \qquad i=1,\ldots,m,
		\]
		so $x_{eq}$ is an equilibrium of \eqref{eq:sdeman}.
		
		On $U$ we have
		\[
		A_i a_i= 2A_iV,
		\qquad i=1,\ldots,m.
		\]
		Substituting into
		\[
		LV = \sum_{i=1}^{m} u_i\,A_iV
		+ \frac{1}{2}\sum_{i=1}^{m}
		\left[a_i^2\,A_i^2V + a_i(A_i a_i)(A_iV)\right]
		\]
		gives
		\begin{align*}
			LV&=\sum_{i=1}^m(-\alpha A_iV)(A_iV)
			+\frac12\sum_{i=1}^m(2V(x))^2 A_i^2V\\
			&\quad+\frac12\sum_{i=1}^m2V(x)(2A_iV)(A_iV)\\
			&=-\alpha\sum_{i=1}^m(A_iV)^2+2V(x)\sum_{i=1}^m(A_iV)^2\\
			&\quad+\frac12(2V(x))^2\sum_{i=1}^m A_i^2V.
		\end{align*}
		
		Let $x=(x_1,\ldots,x_d)^\top$ and
		\begin{align*}
			c_i(x)&:=\bigl(c_i^1(x),\ldots,c_i^d(x)\bigr)^\top\in\mathbb R^d,\\
			G(x)&:=\sum_{i=1}^m c_i(x)c_i(x)^\top.
		\end{align*}
		Since
		\[
		A_iV = \sum_{j=1}^d c_i^j(x)\,x_j = c_i(x)\cdot x,
		\]
		we obtain
		\[
		\sum_{i=1}^m (A_iV)^2 = x^\top G(x)x.
		\]
		By assumption (ii), the vectors $\{A_1(x),\ldots,A_m(x)\}$ span $T_xM$ for
		every $x\in U$, hence $G(x)$ is positive definite for every $x\in U$.
		After shrinking $U$ if necessary, there exist constants $\lambda,\Lambda>0$
		such that
		\[
		\lambda |x|^2 \le x^\top G(x)x \le \Lambda |x|^2
		\qquad \text{for all } x\in U.
		\]
		Since $V(x)=\frac12|x|^2$ on $U$, this implies
		\[
		2\lambda V(x) \le \sum_{i=1}^m (A_iV)^2 \le 2\Lambda V(x)
		\qquad \text{for all } x\in U.
		\]
		Moreover, since $\sum_{i=1}^m A_i^2V$ is smooth on $\overline{U}$, there
		exists a constant $C>0$ such that $
		\left|\sum_{i=1}^m A_i^2V\right| \le C
		$ on  $U.$
		Therefore,
		\begin{align*}
			LV
			&\le -2\alpha\lambda V + 4V(x)\Lambda V + 2V(x)^2 C \\
			&= -2\alpha\lambda V + C_1 V^2,
		\end{align*}
		where $C_1:=4\Lambda+2C$.
		
		Therefore, if $U$ is chosen sufficiently small so that
		\[
		V(x)\le \frac{\alpha\lambda}{C_1}
		\qquad \text{for all } x\in U,
		\]
		then on $U\setminus\{x_{eq}\}$ we have
		\[
		LV \le -\alpha\lambda V < 0.
		\]
		By~\cite[Theorem~5.3]{khasminskii2011stochastic}, this implies that
		$x_{eq}$ is stable in probability. Since $V$ has an infinitesimal upper
		limit, stability in probability is in fact uniform
		(\cite[Remark~5.3]{khasminskii2011stochastic}).
		
		\medskip
		Next we show positive recurrence. Let $W$ be a neighborhood of $x_{eq}$ with
		$W \subset U$, and let $
		D = M \setminus W.$
		The second-order part of $L$ has principal symbol
		\[
		\sigma_x(\xi,\xi)
		=
		\frac12 \sum_{i=1}^m a_i(x)^2 \bigl(g(A_i(x),\xi)\bigr)^2,
		\qquad \xi \in T_xM.
		\]
		
		Since $V^{-1}(0)=\{x_{eq}\}$ and $x_{eq}\notin D$, the function $V$ is strictly
		positive on the compact set $D$. Therefore there exists $\eta>0$ such that
		\[
		V(x)\ge \eta
		\qquad \text{for all } x\in D.
		\]
		If we choose $a_i(x)=2V(x)$ for all $x\in M$, then
		\[
		a_i(x)\ge 2\eta
		\qquad \text{for all } x\in D,\ i=1,\dots,m.
		\]
		
		Consider the function $\sum_{i=1}^m \bigl(g(A_i(x),\xi)\bigr)^2$.
		This is strictly positive away from the equilibrium because
		\[
		\mathrm{span}\{A_1(x),\dots,A_m(x)\}=T_xM
		\qquad \text{for all }x\in D.
		\]
		Hence there exists $\nu>0$ such that
		\[
		\sum_{i=1}^m \bigl(g(A_i(x),\xi)\bigr)^2 \ge \nu |\xi|^2
		\qquad \text{for all }x\in D,\ \xi\in T_xM.
		\]
		Therefore,
		\[
		\sigma_x(\xi,\xi)
		\ge
		\frac{(2\eta)^2}{2}\sum_{i=1}^m \bigl(g(A_i(x),\xi)\bigr)^2
		\ge
		2\eta^2 \nu |\xi|^2.
		\]
		Thus $L$ is uniformly elliptic on all of $D$. By classical elliptic theory, there
		exists a unique bounded solution $u\in C^\infty(\overline{D})$ to the
		Dirichlet boundary value problem
		\[
		\begin{cases}
			Lu(x) = -1, & x\in D, \\[0.4em]
			u(x) = 0,   & x\in \partial W.
		\end{cases}
		\]
		Existence and uniqueness of a weak solution follow by the same
		arguments as in the Euclidean case; see
		\cite[Theorems~1.3 and~1.13]{grigoryan_elliptic_notes}.
		The argument
		extends to the compact Riemannian domain $D$ with smooth boundary in a straightforward manner.	Together with the Sobolev and compact embedding results on manifolds
		\cite[Section~2.10]{aubin2012nonlinear}, these yield solvability.
		Since the data are smooth, standard
		interior and Dirichlet boundary regularity in local charts imply
		$u\in C^\infty(\overline D)$; see~\cite{agranovich2015sobolev}.
		The maximum principle further implies $u\ge0$. Applying Dynkin's formula
		\cite[Theorem~7.4.1]{oksendal2003stochastic} formally to $\tau_W$ gives
		$u(x)=\mathbb{E}_x[\tau_W]\le \|u\|_\infty<\infty$. However, Dynkin's
		formula assumes $\mathbb{E}_x[\tau_W]$ is finite. The argument can be made
		rigorous by applying Dynkin's formula to the bounded stopping time
		$\tau_W\wedge t$ and passing to the limit $t\to\infty$ via monotone
		convergence. Particularly, define the first hitting time of \(W\) by $
		\tau_W:=\inf\{t\ge 0: X_t\in W\},$
		and for \(t>0\) set $
		\sigma_t:=\tau_W\wedge t.$
		Then \(\sigma_t\) is a bounded stopping time, so Dynkin's formula applies:
		\[
		\mathbb E_x[u(X_{\sigma_t})]
		=
		u(x)+\mathbb E_x\int_0^{\sigma_t} Lu(X_s)\,ds.
		\]
		Since \(Lu=-1\) on \(D\), this becomes
		$
		\mathbb E_x[u(X_{\sigma_t})]
		=
		u(x)-\mathbb E_x[\sigma_t],
		$
		or equivalently,$
		\mathbb E_x[\sigma_t]
		=
		u(x)-\mathbb E_x[u(X_{\sigma_t})].
		$
		Because \(u\ge 0\) on \(D\), it follows that $
		\mathbb E_x[\sigma_t]\le u(x).$
		On the other hand, \(\sigma_t=\tau_W\wedge t\) is increasing in \(t\) and
		\[
		\sigma_t\uparrow \tau_W
		\qquad \text{as } t\to\infty.
		\]
		Therefore, by the monotone convergence theorem,
		\[
		\mathbb E_x[\tau_W]
		=
		\lim_{t\to\infty}\mathbb E_x[\tau_W\wedge t]
		\le u(x)<\infty.
		\]
		Hence the hitting time \(\tau_W\) has finite expectation, so \(W\) is positive recurrent.
		
		Combining uniform stability in probability with the positive recurrence of
		$W$ established above, \cite[Theorem~5.7]{khasminskii2011stochastic}
		yields that $x_{eq}$ is asymptotically stable in probability in the large.
	\end{proof}
	
	\textbf{The Partially Underactuated Case}
	
	Assumption \ref{asmp} can be relaxed in the following way allowing one to apply the result with $m=d$ even if the manifold is not parallelizable. Instead of the global assumption, we require only that the vector fields \(\{A_1,\ldots,A_m\}\) satisfy
	\[
	\operatorname{span}\{A_1(x),\ldots,A_m(x)\}=T_xM
	\]
	for all $x$ in a neighborhood of $x_{eq}$,
	so in particular \(m=d\) there, while globally the family \(\{A_1,\ldots,A_m\}\) is assumed to satisfy the \textbf{Lie bracket generating condition}. That is, if \({\rm Lie}\{A_1,\ldots,A_m\}\) denotes the Lie algebra generated by \(A_1,\ldots,A_m\), then
	\[
	\operatorname{span}\{Y(x):Y\in{\rm Lie}\{A_1,\ldots,A_m\}\}=T_xM
	\]
	for all $x\in M$.
	This preserves the local stability argument near \(x_{eq}\). Global recurrence may then be recovered from the bracket-generating condition, through existence of smooth solutions that satisfy the bracket generating condition \cite{bony1969principe}, provided one can choose the neighborhood of the target equilibrium such that it has a {\it noncharacteristic boundary condition}. The assumption of full rank near the equilibrium means we can find a neighborhood of such that the boundary of the neighborhood is noncharacteristic trivially. 
	
	As a basic example to motivate the partially underactuated setting, take \(M=S^2\) and let
	\[
	\Phi:SO(3)\times S^2\to S^2,\qquad \Phi(g,x)=g\cdot x,
	\]
	be the standard action, with $SO(3)$ being the special orthogonal group of matrices. For \(X\in\mathfrak{so}(3)\), we can identify it with \(X\) the associated  vector field
	\[
	X(x):=\left.\frac{d}{dt}\right|_{t=0}\exp(tX)\cdot x.
	\]
	Choose \(X_1,X_2\in\mathfrak{so}(3)\) such that \(X_1,X_2,[X_1,X_2]\) span \(\mathfrak{so}(3)\), and set
	\[
	A_1:=\tilde{X}_1,\qquad A_2:=\tilde{X}_2.
	\]
	We know that 
	\[
	[A_1,A_2]=\tilde{Y}, \qquad Y = [X_1,X_2].
	\]
	Moreover, for each \(x\in S^2\), the differential of the orbit map
	\[
	\Phi_x:SO(3)\to S^2,\qquad \Phi_x(g)=g\cdot x,
	\]
	at the identity is the surjective map
	\[
	d(\Phi_x)_e:\mathfrak{so}(3)\to T_xS^2,\qquad X\mapsto X(x).
	\]
	It follows that
	\[
	\operatorname{span}\{A_1(x),A_2(x),[A_1,A_2](x)\}=T_xS^2
	\]
	for all $x\in S^2$.
	Hence \(\{A_1,A_2\}\) is bracket generating on \(S^2\). Moreover, $X_1$, $X_2$ can be chosen such that 	$
	\operatorname{span}\{A_1(x),A_2(x)\}=T_xS^2$ for all $x$ near the target equilibrium $x_{\rm eq}$.

	The argument that we have provided for the partially underactuated case is only a sketch. A rigorous proof would require extending results on the solvability and boundary regularity of the Dirichlet problem for hypoelliptic operators to the present manifold setting and establishing the corresponding recurrence property. We leave this analysis to future work.
	
	\begin{figure*}[!t]
		\centering   
		\includegraphics[width=0.85\textwidth]{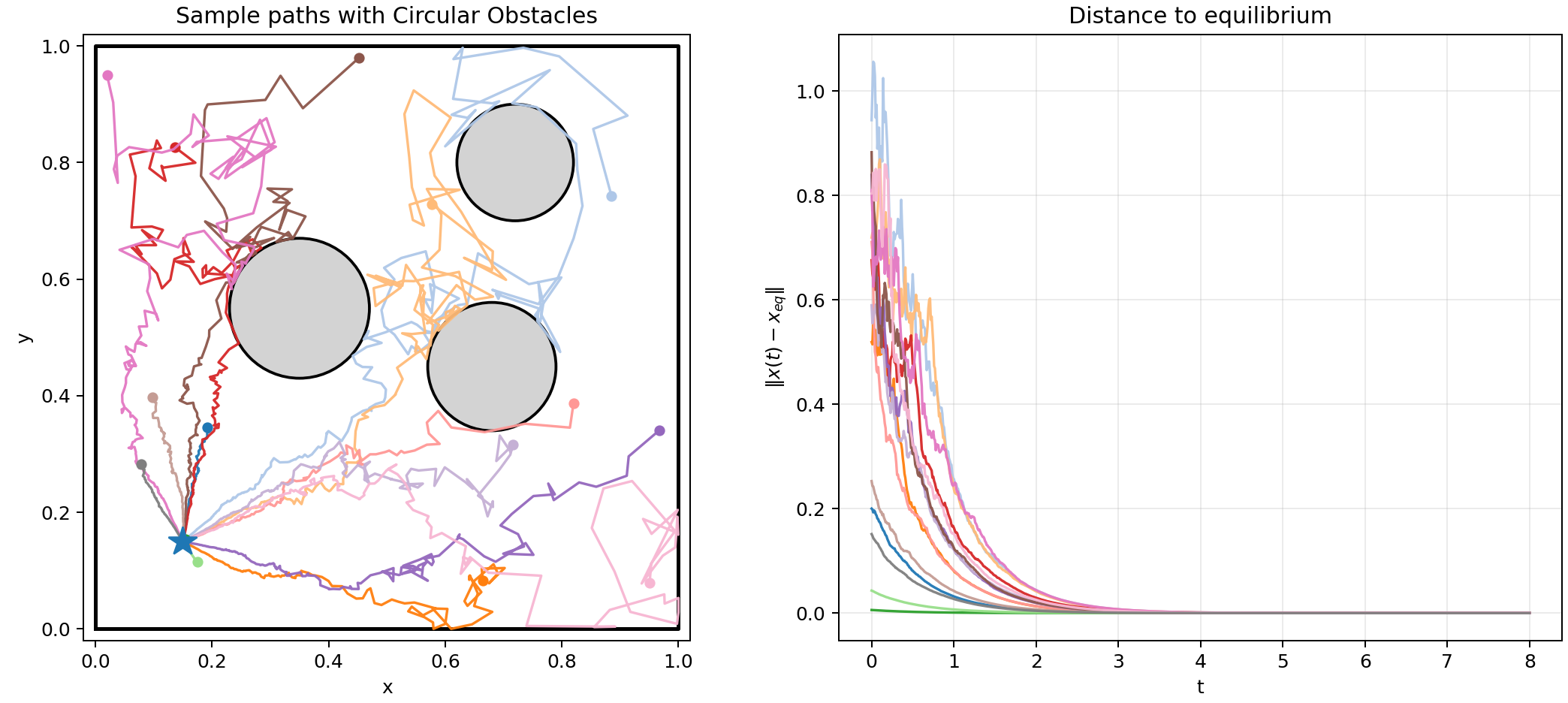}   
		\caption{Sample paths of reflected diffusion with different initial conditions, in the obstacle domain and the corresponding distance-to-equilibrium trajectories.}
		\label{fig:reflected_sde}
	\end{figure*}
	
	\begin{figure*}[!t]
		\centering   
		\includegraphics[width=0.85\textwidth]{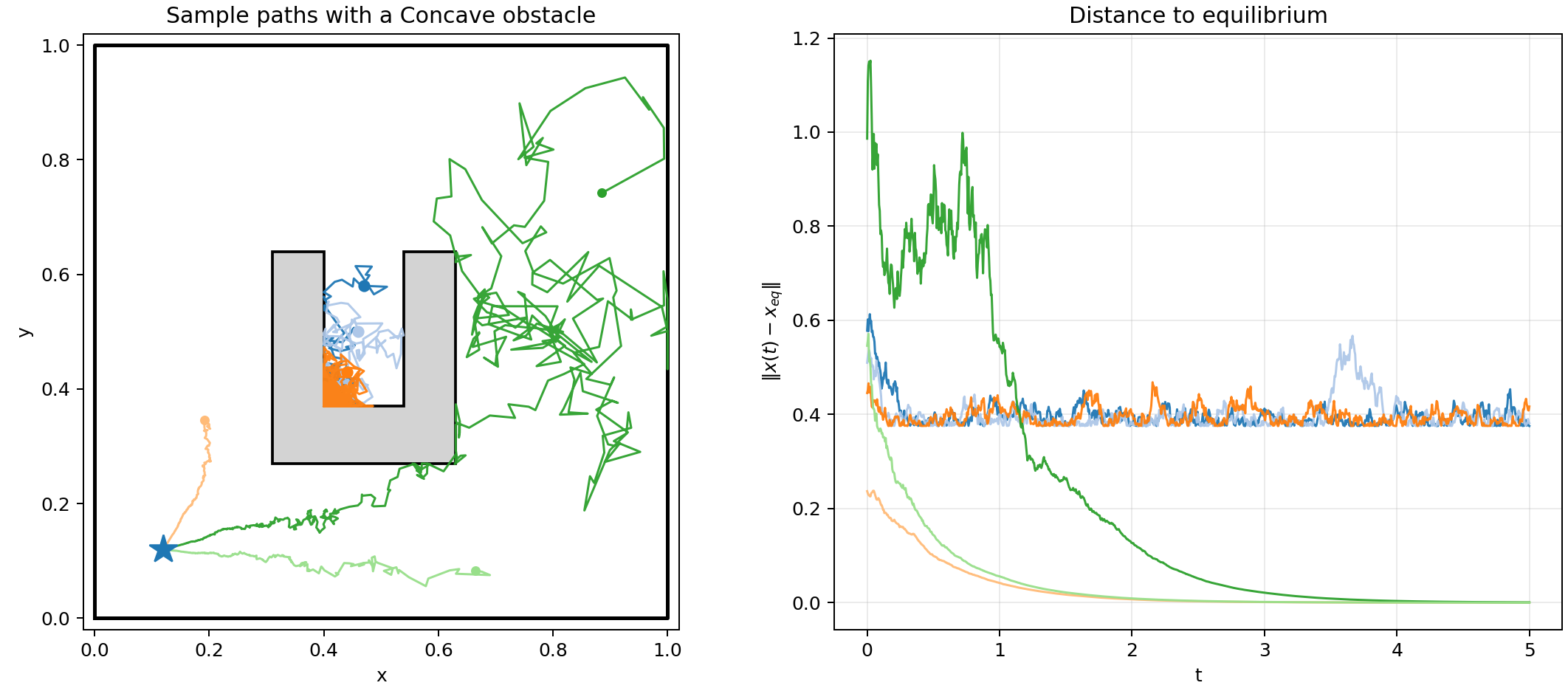}   
		\caption{Sample paths of reflected diffusion with different initial conditions, with a non-convex obstacle and corresponding distance-to-equilibrium trajectories.}
		\label{fig:reflected_sde2}
	\end{figure*}
	
	\begin{figure*}[!t]
		\centering       \includegraphics[width=0.85\textwidth]{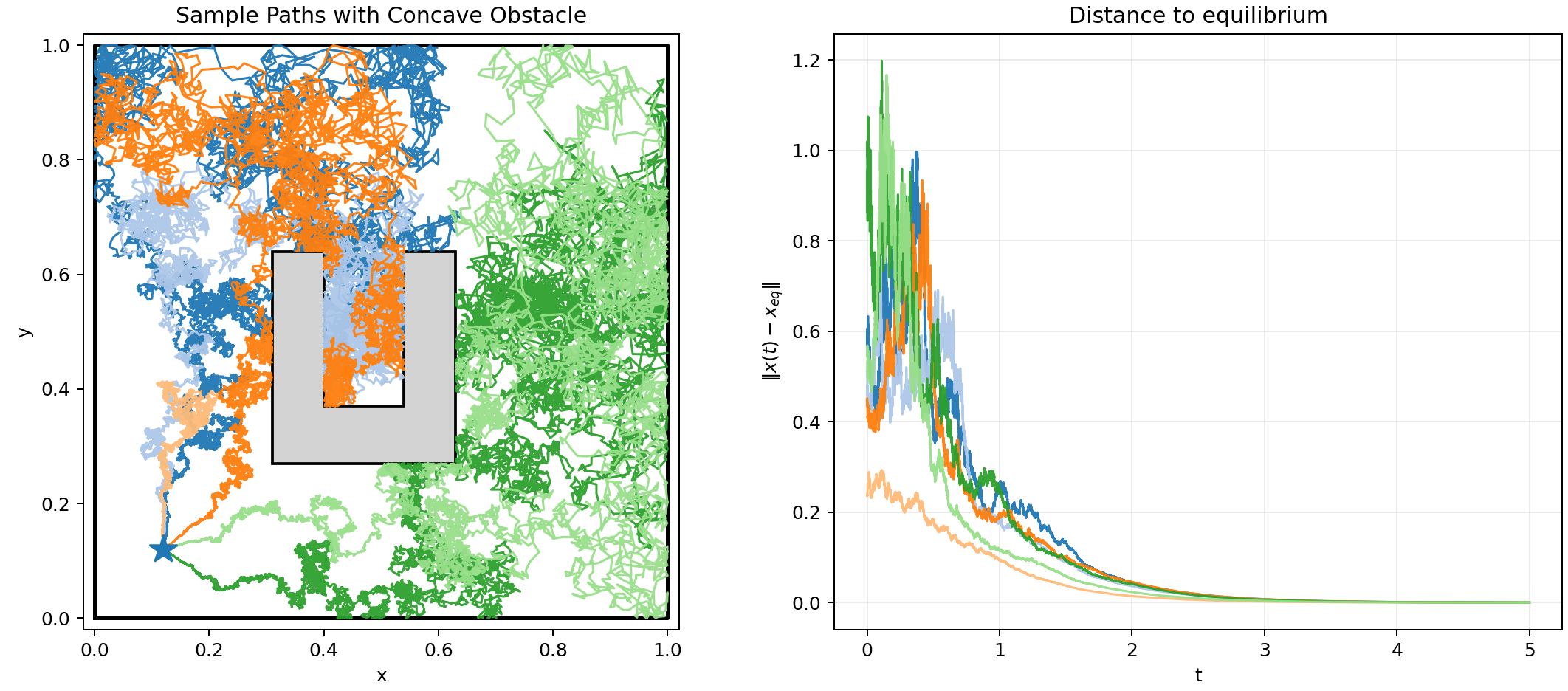}  
		\caption{Sample paths of reflected diffusion with different initial conditions, with a non-convex obstacle and corresponding distance-to-equilibrium trajectories. The amount of noise has been increased to get out of the meta stable region.}
		\label{fig:reflected_sde3}
	\end{figure*}
	
	\begin{figure*}[!t]
		\centering   
		\includegraphics[width=0.90\textwidth]{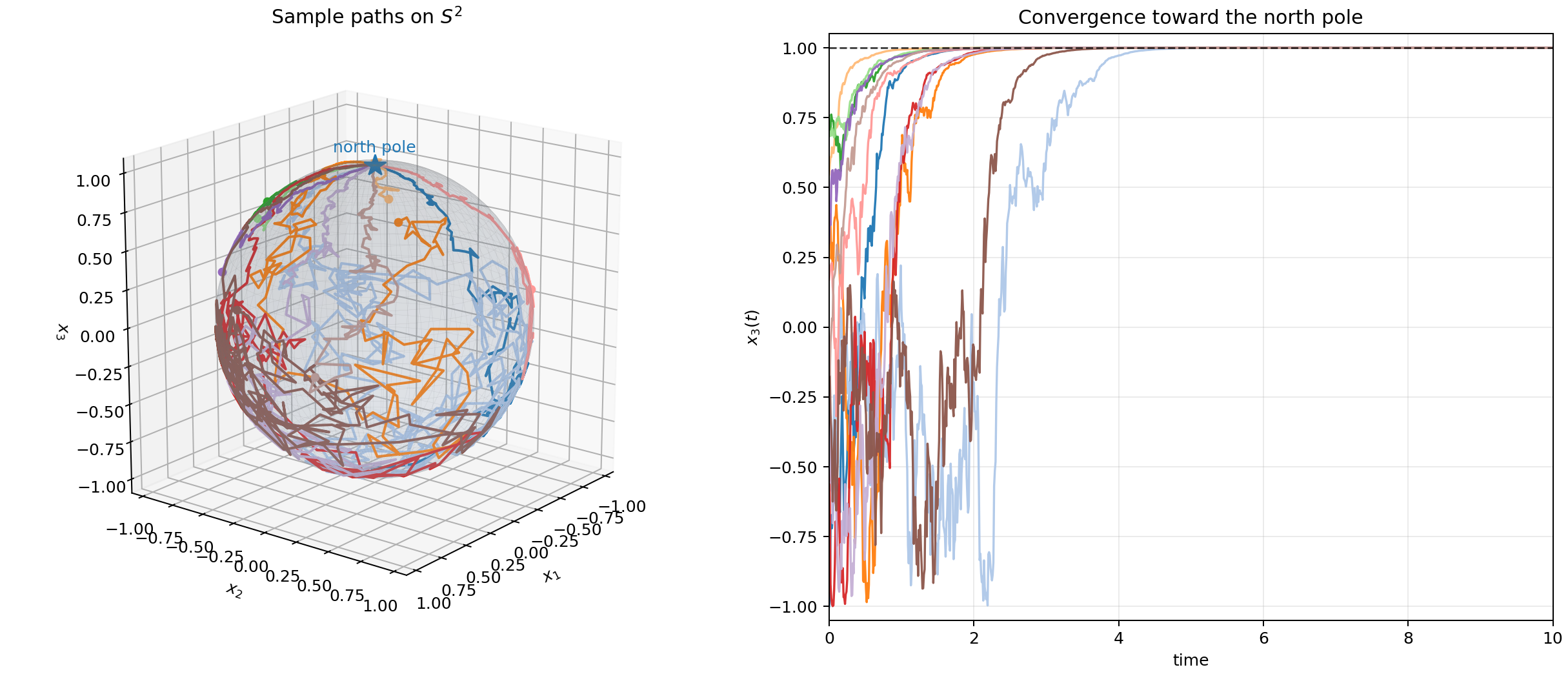}   
		\caption{Sample paths of  diffusion on the sphere, with different initial conditions and the corresponding coordinate $x_3$.}
		\label{fig:sphere}
	\end{figure*}
	
	\section{Numerical Experiments}

	\subsection{Stabilization in a Domain with Obstacles}
	In this section we describe numerical experiments used to illustrate the performance of the reflected diffusion based feedback control approach. First, we consider a bounded planar domain with circular
	obstacles. The free space is taken to be $
	\Gamma = [0,1]^2 \setminus \bigcup_{j=1}^3 \mathcal{O}_j,$
	where $\mathcal{O}_j$ are three circular obstacles, and the target
	equilibrium is chosen as
	\[
	x_{eq} = (0.15,0.15).
	\]
	
	The simulation is based on the choice
	\[
	u_i(x)=-\alpha(x_i-x_{eq,i}),\qquad
	a_i(x)=\sigma|x-x_{eq}|^2,
	\]
	for $i=1,2$,
	where $\alpha>0$ is the drift gain and $\sigma>0$ is a noise-scale
	parameter. In vector form, the dynamics are
	\begin{align*}
		dX(t)&=-\alpha(X(t)-x_{eq})\,dt\\
		&\quad+\sigma\|X(t)-x_{eq}\|^2\circ dW(t)\\
		&\quad+n(X(t))\,d\phi(t),
	\end{align*}
	with reflection enforcing the state constraint $X(t)\in \Gamma$. The results of the experiment can be seen in Figure \ref{fig:reflected_sde}.

	We perform an additional experiment in which the obstacles are concave. This demonstrates the limitations of the stabilizing laws if not chosen appropriately. Though, in theory we know that the target equilibrium is attractive, the presence of non-convex obstacles can induce regions of {\it metastability} in which the the system can spend a long amount of time, before escaping to regions around the equilibrium. This is clearly demonstrated in Figure \ref{fig:reflected_sde2}, where samples initialized at some initial conditions are trapped in metastable regions. The issue of metastability can be reduced to some degree if one increases the amount of noise. This can be seen in Figure \ref{fig:reflected_sde3}

	\subsection{Stabilization on the Sphere}

	In this section we describe the numerical experiment used to illustrate the noise based scheme on a
	manifold - the sphere $S^2 \subset \mathbb{R}^3$ with target
	equilibrium at the north pole $
	x_{eq} = e_3 = (0,0,1).$
	The simulation is based on the coordinate-free choice
	\[
	u_i(x) = -\alpha\,A_iV(x),
	\qquad
	a_i(x)=2V(x),
	\]
	where $V$ is a smooth Lyapunov function with a strict minimum at $x_{eq}$.
	
	We take
	\[
	V(x) = 1-x_3,
	\qquad x=(x_1,x_2,x_3)\in S^2,
	\]
	so that $V\ge 0$ on $S^2$ and $V(x)=0$ if and only if $x=e_3$.
	For the spanning family of tangent vector fields, we take the coordinate vector-fields $(e_1,...e,_3)$ of $\mathbb{R}^d$ and project it onto the tangent space of the sphere to define the vector field
	\[
	A_i(x) = P_x e_i = (I-xx^\top)e_i,
	\qquad i=1,2,3,
	\]
	where $P_x = I-xx^\top$ denotes the orthogonal projection onto the tangent
	space $T_xS^2$.

	These vector fields are smooth, tangent to $S^2$, and satisfy
	\[
	\mathrm{span}\{A_1(x),A_2(x),A_3(x)\} = T_xS^2
	\qquad \text{for all } x\in S^2.
	\]
	The control coefficients are then chosen as
	\[
	u_i(x) = -\alpha\,A_iV(x),
	\qquad
	a_i(x) = 2\sigma V(x),
	\qquad i=1,2,3,
	\]
	where $\alpha>0$ is the drift gain and $\sigma>0$ is a noise-scale parameter.

	Since $\nabla_{\mathbb{R}^3}V = -e_3$, where $\nabla_{\mathbb{R}^3}V $ denotes the Euclidean gradient,  we have
	\[
	A_iV(x) = \nabla_{\mathbb{R}^3}V(x)\cdot A_i(x)
	= -\delta_{i3} + x_i x_3.
	\]
	Hence
	\begin{align*}
		u_1(x)&=-\alpha x_1x_3,\qquad u_2(x)=-\alpha x_2x_3,\\
		u_3(x)&=\alpha(1-x_3^2),
	\end{align*}
	and
	\[
	a_1(x)=a_2(x)=a_3(x)=2\sigma(1-x_3).
	\]
	
	With this choice, the deterministic drift simplifies to
	\[
	\sum_{i=1}^3 u_i(x)A_i(x)
	=
	\alpha(e_3-x_3x),
	\]
	Moreover, the diffusion amplitude vanishes at the north pole since
	\[
	a_i(e_3)=0,
	\qquad i=1,2,3.
	\]

	\section{Conclusion}
	We have shown that, by introducing noise and enforcing domain invariance via reflection, a target equilibrium on a compact Euclidean domain with smooth boundary or on a compact Riemannian manifold can be smoothly stabilized in the sense of stochastic asymptotic stability in the large. The construction circumvents the topological obstructions that preclude smooth deterministic global stabilization, at the cost of adopting a probabilistic notion of convergence. However, as pointed out through numerical experiments, we also advise caution regarding this result. While we have proved a kind of asymptotic stability, the control laws do not necessarily completely deal with the issue of non-convex obstacles due to the issues of metastability. One potential direction of research to address this issue is to minimize the maximum expected time required to visit a neighborhood of the target equilibrium using the related Dirichlet problem, and optimizing for the control laws.
	
	Another natural direction for future work is to extend this strategy to systems with non-holonomic constraints, attempting to strengthen the directions of \cite{nishimura2016stabilization}. While one can still prove recurrence under the bracket generating assumption, the local stability property is harder to attain. Stochastic stabilization of such systems using much weaker notion of stability (in the sense of distributions) has been achieved in \cite{elamvazhuthi2023density}.
	
	\section*{Acknowledgment}
	The author is thankful to Jorge Poveda for helpful comments and suggestions regarding the manuscript.
	
	\section{AI use}
	
	AI was used for developing the code for the simulations and polishing the text. The mathematical discussion and proofs in the paper are due to the author.

	\bibliographystyle{IEEEtran}
	\bibliography{ref}

\end{document}